\documentclass[11pt,a4paper]{amsart}
\usepackage{amsmath,amssymb,amscd,mathrsfs,epic,wasysym,latexsym,tikz,mathrsfs,cite,hyperref}
\usepackage{multirow}

\makeatletter

\numberwithin{equation}{section}
\allowdisplaybreaks

\theoremstyle{plain}
\newtheorem{thm}{Theorem}[section]
\newtheorem{lem}[thm]{Lemma}
\newtheorem{prop}[thm]{Proposition}

\newtheorem{conj}[thm]{Conjecture}

\let\<\langle
\let\>\rangle

\DeclareFontFamily{U}{matha}{\hyphenchar\font45}
\DeclareFontShape{U}{matha}{m}{n}{
      <5> <6> <7> <8> <9> <10> gen * matha
      <10.95> matha10 <12> <14.4> <17.28> <20.74> <24.88> matha12
      }{}
\DeclareSymbolFont{matha}{U}{matha}{m}{n}
\DeclareFontFamily{U}{mathx}{\hyphenchar\font45}
\DeclareFontShape{U}{mathx}{m}{n}{
      <5> <6> <7> <8> <9> <10>
      <10.95> <12> <14.4> <17.28> <20.74> <24.88>
      mathx10
      }{}
\DeclareSymbolFont{mathx}{U}{mathx}{m}{n}

\DeclareMathSymbol{\obot}         {2}{matha}{"6B}
\DeclareMathSymbol{\bigobot}       {1}{mathx}{"CB}

\begin{document}

\title[Theta correspondences]{Coincidence of algebraic and smooth theta correspondences}


\author{YiXin Bao}

\address{Institute of Mathematics, AMSS, CAS, Beijing, 100190, China}

\email{mabaoyixin1984@amss.ac.cn}

\author{Binyong Sun}

\address{Hua Loo-Keng Key Laboratory of Mathematics, AMSS, CAS $\&$ the University of Chinese Academy of Sciences, Beijing, 100190, China} \email{sun@math.ac.cn}

\thanks{B. Sun was  supported by NSFC Grants 11525105, 11321101, and 11531008.}

\begin{abstract}
An ``automatic continuity" question has naturally occurred since Roger Howe established the local theta correspondence over $\mathbb R$: does the algebraic version of local theta correspondence over $\mathbb R$ agrees with the smooth version? We show that the answer is yes, at least when the concerning dual pair has no quaternionic  type I irreducible factor.
\end{abstract}

\maketitle

\section{Introduction}\label{Introduction}

In his seminal work \cite{Ho2}, Roger Howe established the smooth version of local theta correspondence over $\mathbb R$ as a consequence of its algebraic analogue (see \cite[Theorem 1 and Theorem 2.1]{Ho2}). Since then, it has been expected that the smooth version coincides with the algebraic version. Our main goal is to prove this coincidence,  at least when the concerning dual pair has no quaternionic  type I irreducible factor. The proof is a rather direct application of the conservation relations which are established in \cite{SZ}.

To be precise, let $W$ be a finite-dimensional symplectic space over $\mathbb R$ with symplectic form $\langle\, ,\, \rangle_{W}$. An anti-automorphism of the algebra $\mathrm{End}_{\mathbb{R}}(W)$ whose square is the identity map is called an involution on $\mathrm{End}_{\mathbb{R}}(W)$. Denote by $\tau$ the involution of $\mathrm{End}_{\mathbb{R}}(W)$ specified by
\begin{equation}\label{Symplectic Involution}
  \langle x\cdot u, v \rangle_{W}=\langle u, x^\tau\cdot v \rangle_{W}, \quad u,v\in W, \, x\in \mathrm{End}_{\mathbb{R}}(W).
\end{equation}
The involution satisfying \eqref{Symplectic Involution} is called the adjoint involution of  $\langle\, ,\, \rangle_{W}$.  A more  general notion of ``adjoint involution" will be explained in Section \ref{cartaninvo}. Let $(A,A')$ be a pair of $\tau$-stable semisimple real subalgebras of $\mathrm{End}_{\mathbb{R}}(W)$ which are mutual centralizers of each other in $\mathrm{End}_{\mathbb{R}}(W)$. Put
\[
  G:=A\bigcap \mathrm{Sp}(W)\quad\textrm{and}\quad G':=A' \bigcap \mathrm{Sp}(W),
\]
which are closed subgroups of the symplectic group $\mathrm{Sp}(W)$. Following Howe, we call the group pair $(G, G')$ so obtained a reductive dual pair, or a dual pair for simplicity, in $\mathrm{Sp}(W)$.

Write
\begin{equation}\label{Metaplectic Cover for Symplectic Groups}
  1\rightarrow \{1,\varepsilon_{W}\}\rightarrow \widetilde{\mathrm{Sp}}(W)\rightarrow \mathrm{Sp}(W)\rightarrow 1
\end{equation}
for the metaplectic cover of the symplectic group $\mathrm{Sp}(W)$. It does not split unless $W=0$. Denote by $\mathrm{H}(W):= W\times \mathbb{R}$ the Heisenberg group attached to $W$, with the group multiplication
\begin{equation*}
  (u,\alpha)(v,\beta)=(u+v,\alpha+\beta+\langle u,v\rangle_{W}), \quad u,v\in W, \ \alpha,\beta\in \mathbb{R}.
\end{equation*}
Let $\widetilde{\mathrm{Sp}}(W)$ acts on $\mathrm{H}(W)$ as group automorphisms through the action of $\mathrm{Sp}(W)$ on $W$. Then we form the semi-direct product $\widetilde{\mathrm{J}}(W):=\widetilde{\mathrm{Sp}}(W)\ltimes \mathrm{H}(W)$.

When no confusion is possible, we do not distinguish a representation with its underlying space.
Fix an arbitrary non-trivial unitary  character $\psi$ on $\mathbb{R}$. Up to isomorphism, there is a unique smooth Fr\'{e}chet representation (Segal-Shale-Weil representation) $\omega$ of $\widetilde{\mathrm{J}}(W)$ of moderate growth such that (see \cite{Sh} and \cite{We})
\begin{itemize}
  \item $\omega|_{\mathrm{H}(W)}$ is irreducible and has central character $\psi$;
  \item $\varepsilon_{W}\in \widetilde{\mathrm{Sp}}(W)$ acts through the scalar multiplication by -1.
\end{itemize}
We remark that the second condition is automatic unless $W=0$. The representation $\omega$ may be realized on the space of Schwartz functions on a Lagrangian subspace of $W$ (see \cite{Rao}).

For every closed subgroup $E$ of $\mathrm{Sp}(W)$, write $\widetilde{E}$ for the double cover of $E$ induced by the metaplectic cover \eqref{Metaplectic Cover for Symplectic Groups}. Then $\widetilde{G}$ and $\widetilde{G}^{'}$ commute with each other inside the group $\widetilde{\mathrm{Sp}}(W)$ (see \cite[Lemma II.5]{MVW}). Thus the representation $\omega$ induces  a representation of $\widetilde{G}\times \widetilde{G}'$, which we denote by  $\omega_{G, G'}$.

For every  real reductive group $H$, write  $\mathrm{Irr}(H)$ for the set of isomorphism classes of irreducible Casselman-Wallach representations of $H$. Recall that a representation of a real reductive group is called a Casselman-Wallach representation if it is smooth, Fr\'echet, of moderate growth, and its Harish-Chandra module has finite length.  The reader is referred to \cite{Ca}, \cite[Chapter 11]{Wa} or \cite{BK} for
details about Casselman-Wallach representations.

Define
\begin{eqnarray*}
  \mathscr{R}^{\infty}(\widetilde{G},\omega_{G,G'}) &:=& \{\pi\in \mathrm{Irr}(\widetilde{G})|\mathrm{Hom}_{\widetilde{G}}(\omega_{G,G'},\pi)\neq 0\}, \\
  \mathscr{R}^{\infty}(\widetilde{G'},\omega_{G,G'}) &:=& \{\pi'\in \mathrm{Irr}(\widetilde{G'})|\mathrm{Hom}_{\widetilde{G'}}(\omega_{G,G'},\pi')\neq 0\}, \\
 \mathscr{R}^{\infty}(\widetilde{G}\times \widetilde{G'},\omega_{G,G'}) &:=& \{(\pi,\pi')\in \mathrm{Irr}(\widetilde{G})\times \mathrm{Irr}(\widetilde{G'})|\mathrm{Hom}_{\widetilde{G}\times \widetilde{G'}}(\omega_{G,G'},\pi\widehat \otimes \pi')\neq 0\}.
\end{eqnarray*}
Here ``$\widehat \otimes$" indicates the completed projective tensor product. Then the smooth version of the archimedean theta correspondence asserts that $\mathscr{R}^{\infty}(\widetilde{G}\times \widetilde{G'},\omega_{G,G'})$ is the graph of a bijection (smooth theta correspondence) between $\mathscr{R}^{\infty}(\widetilde{G},\omega_{G,G'})$ and $\mathscr{R}^{\infty}(\widetilde{G'},\omega_{G,G'})$.

Now we go to the algebraic version. We say that an involution $\sigma$ on $\mathrm{End}_{\mathbb{R}}(W)$ is a Cartan involution if there is a positive-definite symmetric bilinear form $\langle\,,\,\rangle_\sigma$ on $W$ such that
\begin{equation}\label{Cartan Involution 1}
  \langle x\cdot u, v \rangle_{\sigma}=\langle u,x^\sigma\cdot v \rangle_{\sigma}, \quad u,v\in W, \, x\in \mathrm{End}_{\mathbb{R}}(W).
\end{equation}
The following Theorem will be proved in Section \ref{Standard Cartan Involutions}.

\begin{thm}\label{involutionp}
Up to conjugation by  $G\times G'$, there exists a unique Cartan involution $\sigma$ on $\mathrm{End}_{\mathbb{R}}(W)$ such that
\begin{equation}\label{Cartan Decomposition 2}
  \sigma\circ \tau=\tau\circ \sigma,\quad  \sigma(A)=A\quad\textrm{and}\quad\sigma(A')=A'.
\end{equation}
\end{thm}
The notion of ``conjugation" above will be explained in Section \ref{cartaninvo}.
Fix a Cartan involution $\sigma$ as in Theorem \ref{involutionp}. Then
\[
  C:=\{g\in \mathrm{Sp}(W)\mid g^\sigma g=1\}, \quad K:=G\cap C\quad \textrm{and}\quad  K':=G'\cap C
\]
are maximal compact subgroups of $\mathrm{Sp}(W)$, $G$ and $G'$, respectively. For every $\pi\in \mathrm{Irr}(\widetilde{G})$, define
\[
\mathrm{Hom}_{\widetilde{G}}^{\mathrm{alg}}(\omega_{G,G'},\pi):=\mathrm{Hom}_{\mathfrak g,\widetilde{K}}(\omega_{G,G'}^{\mathrm{alg}},\pi^\mathrm{alg}),
\]
where $\mathfrak g$ denotes the complexified  Lie algebra of $G$, $\omega_{G,G'}^{\mathrm{alg}}$ denotes the $(\mathfrak{sp}(W)\otimes_\mathbb R \mathbb C, \widetilde{C})$-module of $\widetilde{C}$-finite vectors in $\omega=\omega_{G,G'}$, and $\pi^\mathrm{alg}$ denotes the $(\mathfrak g, \widetilde{K})$-module of $\widetilde{K}$-finite vectors in $\pi$. We define $\mathrm{Hom}_{\widetilde{G'}}^{\mathrm{alg}}(\omega_{G,G'},\pi')$ and $\mathrm{Hom}_{\widetilde{G}\times \widetilde{G'}}^{\mathrm{alg}}(\omega_{G,G'},\pi\widehat{\otimes} \pi')$ in an analogous way, where $\pi'\in \mathrm{Irr}(\widetilde{G'})$. The Harish-Chandra module $\omega_{G,G'}^{\mathrm{alg}}$ may be realized as the space of polynomials in $\frac{\dim W}{2}$ variables (see \cite{C}).

In this algebraic setting, we define
\begin{eqnarray*}
  \mathscr{R}^{\mathrm{alg}}(\widetilde{G},\omega_{G,G'}) &:=& \{\pi\in \mathrm{Irr}(\widetilde{G})|\mathrm{Hom}^{\mathrm{alg}}_{\widetilde{G}}(\omega_{G,G'},\pi)\neq 0\}, \\
  \mathscr{R}^{\mathrm{alg}}(\widetilde{G'},\omega_{G,G'}) &:=& \{\pi'\in \mathrm{Irr}(\widetilde{G'})|\mathrm{Hom}^{\mathrm{alg}}_{\widetilde{G'}}(\omega_{G,G'},\pi')\neq 0\}, \\
 \mathscr{R}^{\mathrm{alg}}(\widetilde{G}\times \widetilde{G'},\omega_{G,G'}) &:= & \{(\pi,\pi')\in \mathrm{Irr}(\widetilde{G})\times \mathrm{Irr}(\widetilde{G'})|\mathrm{Hom}^{\mathrm{alg}}_{\widetilde{G}\times \widetilde{G'}}(\omega_{G,G'},\pi\widehat \otimes \pi')\neq 0\}.
\end{eqnarray*}
It is clear that
\begin{equation}\label{inclusion}
\left\{
  \begin{array}{l}
         \mathscr{R}^{\infty}(\widetilde{G},\omega_{G,G'})\subset \mathscr{R}^{\mathrm{alg}}(\widetilde{G},\omega_{G,G'}),\\
    \mathscr{R}^{\infty}(\widetilde{G'},\omega_{G,G'})\subset \mathscr{R}^{\mathrm{alg}}(\widetilde{G'},\omega_{G,G'}),\\
    \mathscr{R}^{\infty}(\widetilde{G}\times \widetilde{G'},\omega_{G,G'})\subset \mathscr{R}^{\mathrm{alg}}(\widetilde{G}\times\widetilde{G'},\omega_{G,G'}).
  \end{array}
\right.
\end{equation}
Then the algebraic version of the archimedean theta correspondence asserts that $\mathscr{R}^{\mathrm{alg}}(\widetilde{G}\times \widetilde{G'},\omega_{G,G'})$ is the graph of a bijection (algebraic theta correspondence) between $\mathscr{R}^{\mathrm{alg}}(\widetilde{G},\omega_{G,G'})$ and $\mathscr{R}^{\mathrm{alg}}(\widetilde{G'},\omega_{G,G'})$. In \cite{Ho2}, Roger Howe  first proves this assertion, and then concludes the smooth version as a consequence of  it.
The following conjecture of the coincidence of the smooth version and the algebraic version of theta correspondence has been expected since the work of Howe.

\begin{conj}\label{Conjecture on Coincidence0}
The inclusions in \eqref{inclusion} are all equalities.
\end{conj}

Recall that the dual pair $(G,G')$ is said to be irreducible of type I if $A$ (or equivalently $A'$) is a simple algebra, and it is said to be irreducible of type II if $A$ (or equivalently $A'$) is the product of two simple algebras which are exchanged by $\tau$. Every dual pair is a direct product in an essentially unique way of irreducible dual pairs, and a complete classification of irreducible dual pairs has been given by Howe \cite{Ho1} (see Section \ref{Standard Cartan Involutions}).

Let $\mathbb H$ denote the quaternionic division algebra over $\mathbb R$. The following theorem is the main result of this paper.

\begin{thm}\label{Main Theorem}
Conjecture $\ref{Conjecture on Coincidence0}$ holds when the dual pair $(G, G')$ contains no quaternionic type I irreducible factor, that is, when $A$ contains no $\tau$-stable ideal which is isomorphic to a matrix algebra $\mathrm M_n(\mathbb H)$ $($$n\geq 1$$)$.
\end{thm}

It is clear that it suffices to prove Theorem \ref{Main Theorem} for irreducible dual pairs. By the classification of irreducible dual pairs, we only need to show that Conjecture $\ref{Conjecture on Coincidence0}$  holds for (real or complex) orthogonal-symplectic dual pairs, unitary dual pairs, and all type II irreducible dual pairs.  These are respectively proved in Sections \ref{orthogonal-symplectic dual pairs}, \ref{seunit} and \ref{sectypeii}.
Section \ref{Standard Cartan Involutions} is devoted to a proof of Theorem \ref{involutionp}.

\section{A proof of Theorem \ref{involutionp}} \label{Standard Cartan Involutions}

\subsection{Cartan involutions}\label{cartaninvo}
To prove Theorem \ref{involutionp}, we explain the notion of Cartan involution on a semisimple algebra over $\mathbb{R}$ and prove several lemmas on Cartan involutions in this section.

Recall that an anti-automorphism of a $\mathbb R$-algebra is called an involution  if its square is the identity map. Here and as usual, all anti-automorphisms of $\mathbb R$-algebras are assumed to be $\mathbb R$-linear. Let $D=\mathbb R, \mathbb C$ or $\mathbb H$, which is a $\mathbb R$-algebra. Let $M$ be a finite-dimensional right $D$-vector space. Given $\epsilon=\pm 1$ and an involution $\iota$ of $D$, an $\iota$-$\epsilon$-Hermitian form on $M$ is a $\mathbb R$-bilinear, non-degenerate map
\begin{equation}\label{hermform}
 \langle\,,\,\rangle:  M\times M\rightarrow D
\end{equation}
such that, for all $m_{1},m_{2}\in M$ and $d\in D$,
\begin{equation*}
  \langle m_{1},m_{2}\cdot d\rangle=\langle m_{1},m_{2}\rangle d, \quad \langle m_{1},m_{2}\rangle=\epsilon\langle m_{2},m_{1}\rangle^{\iota}.
\end{equation*}
The form $\langle\,,\,\rangle$ is called ``$\iota$-Hermitian'' if $\epsilon=1$; and is called ``$\iota$-skew-Hermitian'' if $\epsilon=-1$. A finite-dimensional right  $D$-vector space equipped with an $\iota$-$\epsilon$-Hermitian form is called a right $\iota$-$\epsilon$-Hermitian space.  The notion of $\iota$-$\epsilon$-Hermitian form on a finite-dimensional left $D$-vector space is defined analogously. An involution $\gamma$ on $\mathrm{End}_{D}(M)$ is called the adjoint involution of  the form $\langle\,,\,\rangle$ if
\begin{equation*}
  \langle x\cdot m_{1},m_{2}\rangle=\langle  m_{1},x^{\gamma}\cdot m_{2}\rangle, \quad\textrm{ for all }  m_{1},m_{2}\in M, x\in \mathrm{End}_{D}(M).
\end{equation*}
An $\iota$-$\epsilon$-Hermitian form on $M$ is said to correspond to an involution $\gamma$ on $\mathrm{End}_{D}(M)$ if  $\gamma$ is the adjoint involution of  this form.

We call the $\iota$-$\epsilon$-Hermitian form \eqref{hermform}  an inner product on $M$ if
\begin{itemize}
  \item
 \begin{equation}\label{conjc}
   \iota=\left\{\begin{array}{ll}
                  \textrm{the identity map}, & \textrm{ if $D=\mathbb R$;}\\
                  \textrm{the complex conjugation }, & \textrm{ if $D=\mathbb C$;}\\
                  \textrm{the quaternionic conjugation }, & \textrm{ if $D=\mathbb H$,}
                  \end{array}
                  \right.
\end{equation}
\item
  $\epsilon=1$,  and
\item
$\langle m,m\rangle$ is a positive  real number for all $m\in M\setminus \{0\}$.

\end{itemize}
A Cartan involution $\sigma$ on $\mathrm{End}_{D}(M)$ is defined to be the adjoint involution of  an inner product $\langle\,,\,\rangle_{\sigma}$ on $M$.  Unless otherwise specified, we will use ``$1_\mathbb R$'' and ``$1_\mathbb C$'' to denote the identity map of $\mathbb R$ and $\mathbb C$, respectively; and use ``$\,\overline{\phantom{a} }\,$"  to denote the involution in \eqref{conjc}, namely the identity map on $\mathbb R$, the complex conjugation or the quaternionic conjugation.

Let $A$ be a finite-dimensional semisimple algebra over $\mathbb R$. If $A$ is simple, then by Wedderburn's Theorem (see \cite[Chapter 1, Theorem 1.1]{Ti}), there is an isomorphism
\begin{equation}\label{isoa}
A\cong \mathrm{End}_{D}(M)
\end{equation}
  for some  $M$ as before. We call an involution on $A$ a Cartan involution if it corresponds to a Cartan involution on $\mathrm{End}_{D}(M)$ under the isomorphism \eqref{isoa}. This is independent of the isomorphism.
In general, when $A$ is non-necessarily simple, we say that an involution on $A$ is a Cartan involution if it is a product of Cartan involutions on its simple ideals.

 For any $a\in A^{\times}$, let $\mathrm{Int}(a)$ denote the inner automorphism of $A$ so that
\begin{equation*}
  \mathrm{Int}(a)(x)=axa^{-1}, \quad \forall x\in A.
\end{equation*}
For every involution $\gamma$ on $A$ and  $a\in A^{\times}$, write $a\gamma a^{-1}$ for the involution $\mathrm{Int}(a)\circ \gamma\circ \mathrm{Int}(a^{-1})$, and call it the conjugation of $\gamma$ by $a$.

\begin{lem}\label{Uniquessness of Cartan Involutions}
Up to conjugation by $A^\times$, there exists a unique Cartan involution on $A$.
\end{lem}
\begin{proof}
Without loss of generality, we assume that $A$ is simple and $A=\mathrm{End}_{D}(M)$. Fix an inner product $\langle\,,\,\rangle_\sigma$ on  $M$ and let $\sigma$ be the adjoint involution of  $\langle\,,\,\rangle_\sigma$. Then $a\sigma a^{-1}$ is the adjoint involution of  $\langle\,,\,\rangle_{a\sigma a^{-1}}$, where $a\in A^{\times}$ and $\langle\,,\,\rangle_{a\sigma a^{-1}}$ is the inner product given by
\begin{equation*}
  \langle m_{1},m_{2}\rangle_{a\sigma a^{-1}}:=\langle a^{-1}\cdot m_{1},a^{-1}\cdot m_{2}\rangle_{\sigma}, \quad \forall m_{1},m_{2}\in M.
\end{equation*}
Since all inner products on $M$ are of the form $\langle\,,\,\rangle_{a\sigma a^{-1}}$, we finish the proof.
\end{proof}

The following lemma is  a useful criterion for Cartan involutions.

\begin{lem}\label{Criterion for Cartan Involutions}
Assume that $A$ is simple. Then an involution  $\sigma$ on $A$ is a Cartan  involution if and only if $\mathrm{tr}(xx^{\sigma})$ is a non-negative real number for all $x\in A$. Here ``$\, \mathrm{tr}$" denotes the reduced trace map from $A$ to its center.
\end{lem}
\begin{proof}
For $a\in A^\times$ and $x\in A$,
\begin{eqnarray*}
  \mathrm{tr}(xx^{a\sigma a^{-1}})&=&\mathrm{tr}(xaa^{\sigma}x^{\sigma}(a^{\sigma})^{-1}a^{-1}) \\
   &=&\mathrm{tr}(a^{-1}xaa^{\sigma}x^{\sigma}(a^{\sigma})^{-1}) \\
   &=&\mathrm{tr}((a^{-1}xa)(a^{-1}xa)^{\sigma}).
\end{eqnarray*}
Thus $\sigma$ has the non-negativity property of this lemma if and only if so does $a\sigma a^{-1}$.

 First assume that $(A,M)=(\mathrm{M}_{n}(\mathbb{R}),\mathbb{R}^{n})$. By \cite[Chapter 1, Theorem 4.1]{Ti}, the involution $\sigma$ is the adjoint involution of  a nonsingular symmetric or skew-symmetric form on $M$. Given a non-negative integer $k$, denote by $I_{k}$ the $k\times k$ identity matrix. Up to conjugation, all nonsingular symmetric forms are of the standard form $I_{p,q}=\mathrm{diag}(I_{p},-I_{q})$ such that $p+q=n$, while there is a unique symplectic form if $n$ is even. It is routine to check that the quadratic form $x\mapsto\mathrm{tr}(xx^{\sigma})$ is positive-semidefinite if and only if the corresponding form is the standard form $I_{n,0}$ or $I_{0,n}$, that is, the involution $\sigma$ is a Cartan involution. This proves the lemma in the case when $A\cong \mathrm{M}_{n}(\mathbb{R})$. The proof for the cases of   $A\cong\mathrm{M}_{n}(\mathbb{C})$ and $A\cong \mathrm{M}_{n}(\mathbb{H})$ is similar.
\end{proof}

\begin{lem}\label{Commuting Lemma on Involutions}
Let $\gamma$ be an involution on  $A$. Then up to conjugation by the group $\{a\in A^\times \mid a^\gamma a=1\}$, there exists a unique Cartan involution on $A$ which commutes with $\gamma$.
\end{lem}
\begin{proof}
Without loss of generality, assume that $A$ is either simple or   the product of two simple ideals which are exchanged by $\gamma$. We first treat the latter case by assuming that $A=A_1\times A_2$, where $A_1$ and $A_2$ are two simple ideals of $A$ which are exchanged by $\gamma$. Let $\sigma_1$ be a Cartan involution on $A_1$.  Put $\sigma_2:=\gamma |_{A_1}\circ \sigma_1\circ \gamma |_{A_2}$, which is the unique   involution on $A_2$ such that $\sigma_1\times \sigma_2$ commutes with $\gamma$. Using Lemma \ref{Criterion for Cartan Involutions}, one shows that $\sigma_2$ is also a Cartan involution. This shows the existence assertion of this lemma. The uniqueness assertion is a direct consequence of Lemma \ref{Uniquessness of Cartan Involutions}.

Now we treat the case when $A$ is simple. Assume that  $A=\mathrm{End}_{D}(M)$. By \cite[Chapter 1, Theorem 4.1]{Ti}, $\gamma$ is the adjoint involution of  an $\iota$-$\epsilon$-Hermitian form $\langle\,,\,\rangle_{\gamma}$ on $M$, where $\epsilon=\pm 1$ and $(D,\iota)$ is one of the following four pairs:
\[
  (\mathbb R, 1_\mathbb R), \ (\mathbb C,1_\mathbb C),\ (\mathbb C, \overline{\phantom{a} }),\ (\mathbb H, \overline{\phantom{a} }).
\]
Note that $\gamma$ is simultaneously the adjoint involution of the $\iota$-Hermitian form $\langle\,,\,\rangle_{\gamma}$ and the adjoint involution of the $\iota$-skew-Hermitian form $\mathbf i\langle\,,\,\rangle_{\gamma}$ if $(D,\iota,\epsilon)=(\mathbb C, \overline{\phantom{a} },1)$. Here $\mathbf i:=\sqrt{-1}\in \mathbb C\subset \mathbb H$.

We choose a basis $\{e_{i}\}_{1\leq i\leq n}$ of $M$ such that the matrix $F_{\gamma}:=(\langle e_{i},e_{j}\rangle_{\gamma})_{1\leq i,j\leq n}$ is of the following standard form:
\begin{equation*}
       \left\{
             \begin{array}{ll}
               I_{p,q}, & \hbox{if $(D,\iota)=(\mathbb R, 1_\mathbb R), (\mathbb C, \overline{\phantom{a} })$ or $(\mathbb H, \overline{\phantom{a} })$ and $\epsilon=1$;} \\
               I_{n}, & \hbox{if $(D,\iota)=(\mathbb C, 1_\mathbb C)$ and $\epsilon=1$;} \\
               J_{r}, & \hbox{if $(D,\iota)=(\mathbb R, 1_\mathbb R)$ or $(\mathbb C, 1_\mathbb C)$ and $\epsilon=-1$;} \\
               \mathbf{i}I_{n}, & \hbox{if $(D,\iota)=(\mathbb H, \overline{\phantom{a} })$ and $\epsilon=-1$,}
             \end{array}
           \right.
\end{equation*}
where $p+q=n$ in the first case, and $n=2r$ in the third case. Here $I_{p,q}:=\mathrm{diag}(I_{p},-I_{q})$ and $J_{r}:=\mathrm{diag}(\underbrace{J,J,\cdots,J}_{r})$, where $I_{k}$ is the $k\times k$ identity matrix and $J:=\left(
                                      \begin{array}{cc}
                                        0 & 1 \\
                                        -1 & 0 \\
                                      \end{array}
                                    \right)
$.

We identify $M$ with the column vector space $D^{n}$, hence also $A$ with $\mathrm{M}_{n}(D)$, by means of the basis $\{e_{i}\}_{1\leq i\leq n}$. Fix an inner product $\langle\,,\,\rangle_\sigma$ on $M$ such that $\{e_{i}\}_{1\leq i\leq n}$ is an orthonormal basis for $\langle\,,\,\rangle_\sigma$ and let $\sigma$ be the adjoint Cartan involution of  $\langle\,,\,\rangle_\sigma$. We have that
\begin{equation*}
  ((d_{ij})_{1\leq i,j\leq n})^{\sigma}=(\overline{d_{ji}})_{1\leq i,j\leq n}, \quad ((d_{ij})_{1\leq i,j\leq n})^{\gamma}=F_{\gamma}^{-1}(d_{ji}^{\iota})_{1\leq i,j\leq n}F_{\gamma}.
\end{equation*}
It is routine to check that $\sigma$ commutes with $\gamma$. Consequently, for all $a\in A^\times$ with  $a^\gamma a=1$, the Cartan involution $a\sigma a^{-1}$ also commutes with $\gamma$.

By the Cartan decomposition (see \cite[Theorem 6.31]{Kn}), every element $a$ in $A^{\times}\cong \mathrm{GL}_{n}(D)$ is uniquely of the form
\begin{equation}\label{decoma}
\mathrm{exp}(h)\lambda k,
\end{equation}
 where
  \[
k\in \{a\in A^\times \mid a^\sigma a=1\},\qquad h\in\{a\in A|a=a^{\sigma}, \mathrm{tr}(a)=0\},
 \]
 and $\lambda$ is a positive  real scalar matrix.
Here ``$\mathrm{exp}$" denotes  the exponential map, ``$\mathrm{tr}$" denotes the reduced trace. It is easy to check that $(\lambda k)\sigma(\lambda k)^{-1}=\sigma$. It follows from  Lemma \ref{Uniquessness of Cartan Involutions} that all Cartan involutions on $A$ are of the form $\mathrm{exp}(h)\sigma\mathrm{exp}(-h)$. Together with the non-negative property of $\sigma$ as in Lemma \ref{Criterion for Cartan Involutions}, explicit calculation shows that  the Cartan involution $\mathrm{exp}(h)\sigma\mathrm{exp}(-h)$ commutes with $\gamma$ if and only if
\begin{equation}\label{decoma2}
  \mathrm{exp}(2\iota(h))F_{\gamma}\mathrm{exp}(2h)=F_{\gamma},
\end{equation}
where $\iota((h_{ij})_{1\leq i,j\leq n})=(h_{ji}^{\iota})_{1\leq i,j\leq n}$. By the uniqueness of the decomposition \eqref{decoma}, the equality \eqref{decoma2}  holds if and only if $\iota(h)F_{\gamma}+F_{\gamma}h=0$. This implies that $h$ belongs to the Lie algebra of the group $\{a\in A^\times \mid a^\gamma a=1\}$. Thus $\mathrm{exp}(h)$ belongs to this group and we finish the proof of the lemma.

\end{proof}

Let
\[
(W,\langle\,,\,\rangle_{W}):= \bigoplus_{i=1}^{m}(W_{i},\langle\,,\,\rangle_{W_{i}})
\]
be the orthogonal direct sum of a family $\{(W_{i},\langle\,,\,\rangle_{W_{i}})\}_{1\leq i\leq m}$ of finite-dimensional real symplectic spaces. Denote by $\tau$ the adjoint involution of  $\langle\,,\,\rangle_{W}$ and by $\tau_{i}$ the adjoint involution of  $\langle\,,\,\rangle_{W_{i}}$. It is obvious that the restriction of $\tau$ to $\prod_{i=1}^m \mathrm{End}_{\mathbb{R}}(W_{i})$ is the direct product $\prod_{i=1}^m\tau_{i}$. Let $\sigma_i$ be a Cartan involution on $\mathrm{End}_{\mathbb{R}}(W_{i})$ commuting with $\tau_{i}$ ($1\leq i\leq m$). A Cartan involution $\sigma$ on $\mathrm{End}_{\mathbb{R}}(W)$ is said to be compatible with $\prod_{i=1}^m\sigma_{i}$ if  $(W,\langle\,,\,\rangle_{\sigma})$ is the orthogonal direct sum of  $\{W_{i}\}_{1\leq i\leq m}$, and for each $1\leq i\leq m$, the restriction of $\langle\,,\,\rangle_{\sigma}$ to $W_i\times W_i$ is an inner product corresponding to $\sigma_{i}$, where $\langle\,,\,\rangle_{\sigma}$ is an inner product on $W$ corresponding to $\sigma$, which is unique up to positive scalar multiples. We introduce the following lemma to reduce the proof of Theorem \ref{involutionp} to the cases of irreducible dual pairs.

\begin{lem}\label{Decomposition of Cartan Involutions}
There is a unique Cartan involution $\sigma$ on $\mathrm{End}_{\mathbb{R}}(W)$ which commutes with $\tau$ and  is compatible with $\prod_{i=1}^m\sigma_{i}$.
\end{lem}
\begin{proof}
By the proof of  Lemma \ref{Commuting Lemma on Involutions}, for each $1\leq i\leq m$, there is a basis
\[
\mathcal{B}_{i}=\{e_{k}^{(i)}\}_{1\leq k\leq r_{i}}\cup \{f_{l}^{(i)}\}_{1\leq l\leq r_{i}}
\]
 of $W_{i}$ which is simultaneously a symplectic basis for $\langle\,,\,\rangle_{W_{i}}$ and an orthonormal basis for $\langle\,,\,\rangle_{\sigma_{i}}$, where $\dim W_{i}=2r_{i}$ and $\langle\,,\,\rangle_{\sigma_{i}}$ is an inner product corresponding to $\sigma_{i}$. Hence the union $\mathcal{B}:=\bigcup_{i=1}^{m}\mathcal{B}_{i}$ is a basis of $W$. A Cartan involution $\sigma$ is compatible with $\prod_{i=1}^m\sigma_{i}$ if and only if $(W,\langle\,,\,\rangle_{\sigma})$ is the orthogonal direct sum of $\{(W_{i},\langle\,,\,\rangle_{i})\}_{1\leq i\leq m}$, where $\langle\,,\,\rangle_{i}=a_{i}\langle\,,\,\rangle_{\sigma_{i}}$ ($a_{i}>0$) for each $1\leq i\leq m$. Explicit calculation shows that  $\sigma$ commutes with $\tau$ if and only if all $a_{i}$'s are the same. This shows the existence and the uniqueness of $\sigma$ satisfying the conditions of this lemma.
\end{proof}

\subsection{Irreducible dual pairs}\label{Irreducible dual pairs}

In this section, we recall the classification of irreducible dual pairs. As mentioned in the Introduction,  every irreducible dual pair is either of type I or of  type II.

A family of type I irreducible dual pairs is associated to a pair $(D,\iota)$, which is one of the following four pairs:
\[
  (\mathbb R, 1_\mathbb R), \ (\mathbb C,1_\mathbb C),\ (\mathbb C, \overline{\phantom{a} }),\ (\mathbb H, \overline{\phantom{a} }).
\]
We consider a nonzero right $\iota$-$\epsilon$-Hermitian space $U$ equipped with an $\iota$-$\epsilon$-Hermitian form $\langle\,,\,\rangle_{U}$ and a nonzero left $\iota$-$\epsilon'$-Hermitian space space $V$ equipped with an $\iota$-$\epsilon'$-Hermitian form $\langle\,,\,\rangle_{V}$, where $\epsilon=\pm 1$ and $\epsilon'=-\epsilon$. Then tensor product $W:=U\otimes_{D}V$ is a real symplectic space under the $\mathbb{R}$-bilinear form $\langle\,,\,\rangle_{W}$ such that, for $u_{1},u_{2}\in U$ and $v_{1},v_{2}\in V$,
\begin{equation}\label{Tensor Product of Two Metric Spaces}
  \langle u_{1}\otimes v_{1},u_{2}\otimes v_{2}\rangle_{W}:=\frac{\langle u_{1},u_{2}\rangle_{U}\langle v_{1},v_{2}\rangle_{V}^{\iota}+\overline{\langle u_{1},u_{2}\rangle_{U}\langle v_{1},v_{2}\rangle_{V}^{\iota}}}{2}.
\end{equation}
Hence $(A,A'):=(\mathrm{End}_{D}(U),\mathrm{End}_{D}(V))$ is a pair of $\tau$-stable simple subalgebras of $\mathrm{End}_{\mathbb R}(W)$ which are mutual centralizers of each other in $\mathrm{End}_{\mathbb R}(W)$, where $\tau$ is the adjoint involution of  $\langle\,,\,\rangle_{W}$. Denote by $\mathrm{G}(U)$ the group of all $D$-linear automorphisms of $U$ preserving the form $\langle\,,\,\rangle_{U}$. Define $\mathrm{G}(V)$ analogously. Following \cite{Ho1}, we call $(\mathrm{G}(U),\mathrm{G}(V))=(A\bigcap \mathrm{Sp}(W),A'\bigcap \mathrm{Sp}(W))$ an irreducible dual pair of type I over $(D,\iota,\epsilon)$ as in the Introduction. Sometimes we omit ``$\epsilon$'' and say ``type I irreducible dual pair over $(D,\iota)$''. We list the families of type I irreducible dual pairs in Table \ref{tableg1} and call them real orthogonal-symplectic dual pairs, complex orthogonal-symplectic dual pairs, unitary dual pairs and quaternionic dual pairs, respectively.
\begin{table}[!hbp]
\small
\begin{tabular}{|c|c|}
  \hline
  $(\mathrm{D},\iota)$ &  Dual Pair Family \\
  \hline
   $(\mathbb{R},1_\mathbb R)$  &  $(\mathrm{O}(p,q),\mathrm{Sp}(2r,\mathbb{R}))\subseteq \mathrm{Sp}(2(p+q)r,\mathbb{R})$ \\
  \hline
  $(\mathbb{C},1_\mathbb C)$ & $(\mathrm{O}(p,\mathbb{C}),\mathrm{Sp}(2r,\mathbb{C}))\subseteq \mathrm{Sp}(4pr,\mathbb{R})$ \\
  \hline
  $(\mathbb{C},\overline{\phantom{a} })$ &  $(U(p,q),U(r,s))\subseteq \mathrm{Sp}(2(p+q)(r+s),\mathbb{R})$ \\
  \hline
   $(\mathbb{H},\overline{\phantom{a} })$ &  $(\mathrm{Sp}(p,q),\mathrm{O}^{*}(2r))\subseteq \mathrm{Sp}(4(p+q)r,\mathbb{R})$ \\
  \hline
\end{tabular}
\vspace{0.1cm}
\caption{\small Type I Irreducible Dual Pairs}\label{tableg1}
\end{table}

As described in \cite{Ho2}, a family of type II irreducible dual pairs is associated to a division algebra $D$ over $\mathbb R$. We consider two nonzero right $D$-vector spaces $U_{1}$ and $U_{2}$ of dimension $r$ and $s$, respectively. Put
\[
  U_{1}^{*}:= \mathrm{Hom}_{D}(U_{1},D)\quad \textrm{and} \quad  U_{2}^{*}:= \mathrm{Hom}_{D}(U_{2},D),
  \]
   where the division algebra $D$ is viewed as a right $D$-vector space. Note that $U_{1}^{*}$ and $U_{2}^{*}$ are naturally left $D$-vector spaces. Set
   \begin{equation}\label{decomwxy}
   W=X\oplus Y, \qquad\textrm{where $X=U_{1}\otimes_{D}U_{2}^{*} \ $ and $\ Y=U_{2}\otimes_{D}U_{1}^{*}$. }
   \end{equation}
   There is a symplectic form $\langle\,,\,\rangle_{W}$ on $W$ such that $X$ and $Y$ are two maximal isotropic subspaces, and for all $u_{1}\in U_{1}, u_{2}\in U_{2}, u_{1}^{*}\in U_{1}^{*}$ and $u_{2}^{*}\in U_{2}^{*}$,
\begin{equation*}
  \langle u_{1}\otimes u_{2}^{*}, u_{2}\otimes u_{1}^{*}\rangle_{W}:= \frac{(u_{1}^{*},u_{1})(u_{2}^{*},u_{2})+\overline{(u_{1}^{*},u_{1})(u_{2}^{*},u_{2})}}{2},
\end{equation*}
where $(u_{1}^{*},u_{1})$ and $(u_{2}^{*},u_{2})$ are the canonical pairing with dual spaces. Put
\[
A:= \mathrm{End}_{D}(U_{1})\oplus \mathrm{End}_{D}(U_{1}^{*})\quad \textrm{and} \quad A':= \mathrm{End}_{D}(U_{2})\oplus \mathrm{End}_{D}(U_{2}^{*}).
\]
 Then $(A,A')$ is a pair of $\tau$-stable subalgebras of $\mathrm{End}_{\mathbb{R}}(W)$ which are mutual centralizers of each other, where $\tau$ is the adjoint involution of  $\langle\,,\,\rangle_{W}$.  Moreover,  the simple algebras $\mathrm{End}_{D}(U_{1})$  and $\mathrm{End}_{D}(U_{1}^{*})$ are exchanged by $\tau$, and likewise $\mathrm{End}_{D}(U_{2}^{*})$  and $\mathrm{End}_{D}(U_{2})$ are also exchanged by $\tau$. Denote by $\mathrm{GL}(U_{1})$ the group of all $D$-linear automorphisms of $U_{1}$ and define $\mathrm{GL}(U_{2})$ analogously. The groups $A\bigcap \mathrm{Sp}(W)$ and $A'\bigcap \mathrm{Sp}(W)$ are isomorphic to $\mathrm{GL}(U_{1})$ and $\mathrm{GL}(U_{2})$, respectively. We call $(\mathrm{GL}(U_{1}),\mathrm{GL}(U_{2}))$ an irreducible dual pair of type II as in the Introduction. The reader is referred to \cite[Chapter II]{Ku2} for details.  We list the families of type II irreducible dual pairs in Table \ref{tableg2}.
\begin{table}[!h]
\small
  \begin{tabular}{|c|c|}
  \hline
    $\mathrm{D}$ & Dual Pair Family \\
  \hline
   $\mathbb{R}$ & $(\mathrm{GL}_{r}(\mathbb{R}),\mathrm{GL}_{s}(\mathbb{R}))\subseteq \mathrm{Sp}(2rs,\mathbb{R})$  \\
  \hline
   $\mathbb{C}$  & $(\mathrm{GL}_{r}(\mathbb{C}),\mathrm{GL}_{s}(\mathbb{C}))\subseteq \mathrm{Sp}(4rs,\mathbb{R})$ \\
  \hline
   $\mathbb{H}$  &  $(\mathrm{GL}_{r}(\mathbb{H}),\mathrm{GL}_{s}(\mathbb{H}))\subseteq \mathrm{Sp}(8rs,\mathbb{R})$ \\
  \hline
  \end{tabular}
\vspace{0.1cm}
\caption{\small Type II Irreducible Dual Pairs}\label{tableg2}
\end{table}

\subsection{The proof of Theorem \ref{involutionp}} Let $A, A'\subset \mathrm{End}_{\mathbb{R}}(W)$ and  $G,G'\subset \mathrm{Sp}(W)$ be as in  Theorem \ref{involutionp}. Recall that $\tau$ is the adjoint involution of  the symplectic form $\langle\,,\,\rangle_{W}$. We want to to show that
up to conjugate by  $G\times G'$, there exists a unique Cartan involution $\sigma$ on $\mathrm{End}_{\mathbb{R}}(W)$ such that
\begin{equation}\label{Cartan Decomposition 3}
  \sigma\circ \tau=\tau\circ \sigma,\quad  \sigma(A)=A\quad\textrm{and}\quad\sigma(A')=A'.
\end{equation}

First we assume that $(G,G')$ is irreducible. A Cartan involution satisfying \eqref{Cartan Decomposition 3} can be explicitly constructed (see \cite[Section 1]{AB1} and \cite[Section 1.1]{Pa1}). We only need to prove that $\sigma$ is unique up to conjugation by $G\times G'$.

\begin{prop}\label{involutionp for type I case}
The uniqueness assertion of Proposition $\ref{involutionp}$ holds for  irreducible dual pairs of type I.
\end{prop}
\begin{proof}
Given a type I irreducible dual pair $(\mathrm{G}(U),\mathrm{G}(V))$ over $(D,\iota)$, recall  the real symplectic space $W$, the algebra pair $(A,A')$ and the involution $\tau$ from  Section \ref{Irreducible dual pairs}. Suppose $\sigma$ is a Cartan involution on $W=U\otimes_{D}V$ satisfying \eqref{Cartan Decomposition 3} and $\langle\,,\,\rangle_{\sigma}$ is an inner product corresponding to $\sigma$. Note that $(A, A')$ is $\sigma$-stable and $\tau$-stable. Denote by $\sigma_{1}$ and $\sigma_{2}$ the restrictions of $\sigma$ on $A$ and $A'$, respectively. Define $\tau_{1}$ and $\tau_{2}$ analogously. By Lemma \ref{Criterion for Cartan Involutions} and \eqref{Cartan Decomposition 3}, $\sigma_{1}$ is a Cartan involution on $\mathrm{End}_{D}(U)$ and commutes with $\tau_{1}$. Likewise, $\sigma_{2}$ is a Cartan involution on $\mathrm{End}_{D}(V)$, and commutes with $\tau_{2}$, respectively.

Let $\langle\,,\,\rangle_{2}$ be an inner product on $V$ corresponding to $\sigma_{2}$.
Define a norm on $U$ by
 \begin{equation}\label{defabsu}
  \| u \| :=\sqrt{\langle u\otimes v_{0}, u \otimes v_{0}\rangle_{\sigma}},
 \end{equation}
where $v_{0}$ is a vector in $V$ satisfying $\langle v_{0},v_{0}\rangle_{2}=1$. Since the set of unit vectors in $V$ is transitive under the action of the group $\{y\in \mathrm{End}_{D}(V)|yy^{\sigma_{2}}=1\}$, which is a subgroup of $\{a\in \mathrm{End}_{\mathbb{R}}(W)|aa^{\sigma}=1\}$, this norm is independent of the choice of the unit vector $v_0$. Note that
\[
\|u\cdot d\|=\|u\|\cdot|d|\quad\textrm{for all  $d\in D$ and $u\in U$,$\ $ where $|d|:=\sqrt{d\overline{d}}$.}
\]
 Moreover, this norm satisfies parallelogram law (see \cite[Chapter V, Theorem 2]{J}). Hence there is a unique  inner product $\langle\,,\,\rangle_{1}$ on $U$ such that (see \cite[Chapter V, Theorem 4]{J})
  \begin{equation}\label{defabsu2}
   \langle u, u\rangle_{1}=\|u \|^2, \quad \textrm{for all }u\in U.
 \end{equation}
 It is routine to check that $\langle\,,\,\rangle_{1}$ is an inner product corresponding to $\sigma_{1}$. By \eqref{defabsu} and \eqref{defabsu2},
\begin{equation}\label{defabsu3}
    \langle u\otimes v,u\otimes v\rangle_{\sigma}=\langle u,u\rangle_{1} \cdot \langle v,v\rangle_{2},\quad \textrm{for all }u\in U, \, v\in V.
     \end{equation}

 Define a bilinear form on $W$ by
 \begin{equation}\label{productf}
  \langle u_{1}\otimes v_{1},u_{2}\otimes v_{2}\rangle_{\sigma}'=\frac{\langle u_{1},u_{2}\rangle_{1}\overline{\langle v_{1},v_{2}\rangle_{2}}+\langle v_{1},v_{2}\rangle_{2}\overline{\langle u_{1},u_{2}\rangle_{1}}}{2}.
\end{equation}
 One checks that it is  a well-defined non-degenerate symmetric bilinear form on $W$ corresponding to $\sigma$.  Thus this form is a scalar multiple of $\langle\,,\,\rangle_\sigma$. It must equal $\langle\,,\,\rangle_\sigma$ by \eqref{defabsu3}. In conclusion,  the inner product $\langle\,,\,\rangle_{\sigma}$ is determined by the inner products $\langle\,,\,\rangle_{1}$ and $\langle\,,\,\rangle_{2}$, and consequently,  $\sigma$ is determined by $\sigma_{1}$ and $\sigma_{2}$. Therefore this lemma follows by Lemma \ref{Commuting Lemma on Involutions}.
 \end{proof}

Recall from  Section \ref{Irreducible dual pairs} the type II irreducible dual pair $(\mathrm{GL}(U_{1}),\mathrm{GL}(U_{2}))$, the decomposition  $W=X\oplus Y$ of \eqref{decomwxy},  the algebra pair $(A,A')$ and the involution $\tau$.   Let  $\sigma$ be a Cartan involution on $W$ satisfying \eqref{Cartan Decomposition 3} and $\langle\,,\,\rangle_{\sigma}$ an inner product corresponding to $\sigma$. To prove the uniqueness of $\sigma$, we introduce the following lemma.

\begin{lem}\label{Orthogonal Direct Sum Decomposition}
The decomposition $W=X\oplus Y$ is an orthogonal direct sum decomposition of  $\langle\,,\,\rangle_{\sigma}$.
\end{lem}
\begin{proof}
Let $p_{X}$  denote the projection of $W$ onto $X$ with respect to the decomposition $W=X\oplus Y$. It is an  element of  the center $Z$ of $A$. It is clear that $Z$ is $\sigma$-stable. Hence
\[
Z=Z^{+}\oplus Z^{-},\quad  \textrm{where $Z^{+}=\{z\in Z|z^{\sigma}=z\}$ and $Z^{-}=\{z\in Z|z^{\sigma}=-z\}$. }
\]
For each $z\in Z^-$, $z^\sigma=-z$ implies that all the eigenvalues of  $z\in \mathrm{End}_\mathbb R(W)$ are pure imaginary numbers. Likewise all eigenvalues of all elements of  $ Z^{+}$ are real numbers. Since all the eigenvalues of $p_X$ are real numbers, $p_X\in Z^+$, that is, $p_{X}$ is fixed by $\sigma$.  Therefore,  for all $w_{X}\in X$ and $w_{Y}\in Y$,
\begin{equation*}
  \langle w_{X},w_{Y}\rangle_{\sigma}=\langle p_{X}\cdot w_{X},w_{Y}\rangle_{\sigma}=\langle w_{X},p_{X}^{\sigma}\cdot w_{Y}\rangle_{\sigma}=\langle w_{X},p_{X}\cdot w_{Y}\rangle_{\sigma}=0.
\end{equation*}
This  finishes the proof.
\end{proof}

\begin{prop}\label{involutionp for type II case}
The uniqueness assertion of Proposition $\ref{involutionp}$ holds for  irreducible dual pairs of type II.
\end{prop}
\begin{proof}
By Lemma \ref{Orthogonal Direct Sum Decomposition}, the decomposition  $W=X\oplus Y$ of \eqref{decomwxy} is an orthogonal direct sum decomposition of $\langle\,,\,\rangle_{\sigma}$, while $X$ and $Y$ are two maximal totally isotropic subspaces of $W$ with respect to $\langle\,,\,\rangle_{W}$. Hence the semisimple algebra $\mathrm{End}_{\mathbb{R}}(X)\times \mathrm{End}_{\mathbb{R}}(Y)$ is $\sigma$-stable and $\tau$-stable. Denote by $\widetilde{\sigma}$ and $\widetilde{\tau}$ the restrictions of $\sigma$ and $\tau$ on $\mathrm{End}_{\mathbb{R}}(X)\times \mathrm{End}_{\mathbb{R}}(Y)$, respectively. Thus $(\mathrm{End}_{\mathbb{R}}(X),\mathrm{End}_{\mathbb{R}}(Y))$ is a pair of $\widetilde{\sigma}$-stable simple subalgebras of $\mathrm{End}_{\mathbb{R}}(W)$, which are exchanged by $\widetilde{\tau}$.  By Lemma \ref{Criterion for Cartan Involutions}, $\widetilde{\sigma}$ is a Cartan involution on $\mathrm{End}_{\mathbb{R}}(X)\times \mathrm{End}_{\mathbb{R}}(Y)$. As in the proof of Lemma \ref{Decomposition of Cartan Involutions}, $\sigma$ is determined by $\widetilde{\sigma}$, up to conjugation by
\begin{equation*}
  Z_{\mathbb R^{\times}}:=\{ap_{X}+a^{-1}p_{Y}|a\in \mathbb{R}^{\times}\},
\end{equation*}
where $p_{X}$ and $p_{Y}$ are the projections of $W$ onto $X$ and $Y$ with respect to the decomposition $W=X\oplus Y$, respectively. By the same argument as in the proof of Lemma \ref{Commuting Lemma on Involutions}, the Cartan involution $\widetilde{\sigma}=\sigma_{X}\times \sigma_{Y}$ is determined by $\sigma_{X}$, where $\sigma_{X}$ and $\sigma_{Y}$ are the restrictions of $\sigma$ on $\mathrm{End}_{\mathbb{R}}(X)$ and $\mathrm{End}_{\mathbb{R}}(Y)$, respectively. This amounts to saying that $\sigma$ is determined by $\sigma_{X}$, up to conjugation by $Z_{\mathbb R^{\times}}$. Note that $Z_{\mathbb R^{\times}}$ is a subgroup of $Z^{\times}\bigcap \mathrm{Sp}(W)$, where $Z$ is the center of both $A$ and $A'$.

As in the proof of Proposition \ref{involutionp for type I case}, let $\sigma_{1}$ and  $\sigma_{2}$ be the restrictions of $\sigma$ on $A$ and $A'$, respectively. Define  $\tau_{1}$ and $\tau_{2}$ analogously.  Recall that
\[
A=A_{X}\oplus A_{Y},\quad  \textrm{where $A_{X}=\mathrm{End}_{D}(U_{1})$ and $A_{Y}=\mathrm{End}_{D}(U_{1}^{*})$, }
\]
\[
A'=A'_{X}\oplus A'_{Y},\quad  \textrm{where $A'_{X}=\mathrm{End}_{D}(U_{2}^{*})$ and $A'_{Y}=\mathrm{End}_{D}(U_{2})$. }
\]
By the same argument as in the proof of Lemma \ref{Commuting Lemma on Involutions}, $\sigma_{1}$ and $\sigma_{2}$ are determined by $\sigma_{1}|_{A_X}$ and $\sigma_{2}|_{A'_X}$, respectively. By the same argument as in the proof of Proposition \ref{involutionp for type I case}, the restriction of $\langle\,,\,\rangle_{\sigma}$ on $X\times X$ is determined by $\langle\,,\,\rangle_{1}$, which is an inner product on $U_{1}$ corresponding to $\sigma_{1}|_{A_X}$, and $\langle\,,\,\rangle_{2}$, which is an inner product on $U_{2}^{*}$ corresponding to $\sigma_{2}|_{A'_X}$. Consequently, $\sigma_{X}$ is determined by $\sigma_{1}|_{A_X}$ and $\sigma_{2}|_{A'_X}$. Therefore this lemma follows by Lemma \ref{Uniquessness of Cartan Involutions}.
\end{proof}

In general, the dual pair $(G,G')$ is the direct product of a family $\{(G_{i},G_{i}')\}_{1\leq i\leq m}$ of irreducible dual pairs,  where $(G_{i},G_{i}')$ is a dual pair in  $\mathrm{Sp}(W_{i})$, and
\begin{equation} \label{decw}
 W=\bigoplus_{i=1}^m W_i
\end{equation}
is an orthogonal direct sum of symplectic spaces. Let $(A_{i}, A_{i}')$ be as in Section \ref{Irreducible dual pairs}. Let ${Z}_{i}$ ($1\leq i\leq m$) be the center of $A_{i}$, which also equals to the center of   $A_{i}'$. Then $(A,A')$ is the direct sum of $\{(A_{i},A_{i}')\}_{1\leq i\leq m}$ and
\[
   Z:=\bigoplus_{i=1}^m Z_i
\]
is the center of both $A$ and $A'$. Denote by $\tau_{i}$ ($1\leq i\leq m$) the restriction of $\tau$ on $\mathrm{End}_{\mathbb{R}}(W_{i})$. By Proposition \ref{involutionp for type I case} and Proposition \ref{involutionp for type II case}, there exists a Cartan involution $\sigma_{i}$ on $\mathrm{End}_{\mathbb{R}}(W_{i})$ satisfying
\begin{equation*}
  \sigma_{i}\circ \tau_{i}=\tau_{i}\circ \sigma_{i}, \quad \sigma_{i}(A_{i})=A_{i}, \quad \sigma_{i}(A_{i}')=A_{i}',
\end{equation*}
for each $1\leq i\leq m$. By Lemma \ref{Decomposition of Cartan Involutions}, there exists a Cartan involution $\sigma$ on $\mathrm{End}_{\mathbb{R}}(W)$ satisfying \eqref{Cartan Decomposition 3}. We only need to prove the uniqueness of $\sigma$.

Let $\sigma$ be a Cartan involution satisfying \eqref{Cartan Decomposition 3} and $\langle\,,\,\rangle_{\sigma}$ an inner product corresponding to $\sigma$. For each $1\leq i\leq m$, let  $p_{i}\in \mathrm{End}_\mathbb R(W)$ denote  the projection of $W$ onto $W_{i}$ with respect to the decomposition \eqref{decw}. It is an element of ${Z}_{i}\subset Z$. By the same argument as in the proof of Lemma \ref{Orthogonal Direct Sum Decomposition}, we show that $p_{i}$ is fixed by $\sigma$. Consequently,  \eqref{decw} is an orthogonal decomposition with respect to $\langle\,,\,\rangle_{\sigma}$, and $\mathrm{End}_{\mathbb{R}}(W_{i})$ is $\sigma$-stable for all $1\leq i\leq m$. By Lemma \ref{Decomposition of Cartan Involutions}, the Cartan involution $\sigma$ is determined by $\prod_{i=1}^{m}\sigma_{i}$, where $\sigma_{i}$  is the restriction of $\sigma$ on $\mathrm{End}_{\mathbb{R}}(W_{i})$. By Proposition \ref{involutionp for type I case} and Proposition \ref{involutionp for type II case}, $\sigma_{i}$ is unique up to conjugation by $G_{i}\times G_{i}'$, for each $1\leq i\leq m$. Therefore $\sigma$ is unique up to conjugation by $G\times G'$. This finishes the proof of Theorem \ref{involutionp}.

\section{Type II irreducible dual pairs}\label{sectypeii}

Let the notation be as in the Introduction. Recall that for each closed subgroup $E$ of $\mathrm{Sp}(W)$, $\widetilde E$ denotes its double cover induced by the metaplectic cover $\widetilde{\mathrm{Sp}}(W)\rightarrow \mathrm{Sp}(W)$. If $E$ is reductive as a Lie group,  then so is $\widetilde E$, and we write $\mathrm{Irr}^\mathrm{gen}(\widetilde E)$ for the subset of $\mathrm{Irr}(\widetilde E)$ of genuine irreducible Casselman-Wallach representations. Here and as usual, ``genuine" means that the representation does not descend to a representation of $E$.

Note that by the one-one correspondence property of local theta correspondence, if one of the three inclusions in \eqref{inclusion} is an equality, then so are the other two.
We first treat the simplest case of type II irreducible dual pairs.

\begin{prop}\label{typeii}
Conjecture \ref{Conjecture on Coincidence0} holds for all type II irreducible dual pairs.
\end{prop}
\begin{proof}
Suppose that  $(G,G')=(\mathrm{GL}_{r}(D),\mathrm{GL}_{s}(D))$ is a type II irreducible dual pair. Without loss of generality, assume that $r\leq s$.  By Godment-Jacquet zeta integrals \cite[Section I.8]{GJ} and Kudla's persistence principle  \cite{Ku1}, we know that all representations in  $\mathrm{Irr}^\mathrm{gen}(\widetilde G)$ occur in the smooth theta correspondence, namely,
\[
      \mathrm{Irr}^\mathrm{gen}(\widetilde G)\subset   \mathscr{R}^{\infty}(\widetilde{G},\omega_{G,G'}).
\]
Since
\[
   \mathscr{R}^{\infty}(\widetilde{G},\omega_{G,G'})\subset \mathscr{R}^{\mathrm{alg}}(\widetilde{G},\omega_{G,G'})\subset \mathrm{Irr}^\mathrm{gen}(\widetilde G),
\]
we conclude that $ \mathscr{R}^{\infty}(\widetilde{G},\omega_{G,G'})= \mathscr{R}^{\mathrm{alg}}(\widetilde{G},\omega_{G,G'})$. This proves the proposition.
\end{proof}

Similar proof shows the following result for quaternionic  dual pairs.

\begin{prop}
Assume that $(G, G')=(\mathrm{G}(U),\mathrm{G}(V))$ is a quaternionic dual pair, where $U$ is a quaternionic Hermitian space. If $\dim U\leq \dim V$, then the three inclusions in \eqref{inclusion}  are all equalities.
\end{prop}
\begin{proof}
By \cite[Theorem 7.3]{SZ},
\[
      \mathrm{Irr}^\mathrm{gen}(\widetilde G) \subset \mathscr{R}^{\infty}(\widetilde{G},\omega_{G,G'}).
\]
Then the proposition follows by the  same argument as in the proof of Proposition \ref{typeii}.
\end{proof}

\section{Orthogonal-symplectic dual pairs}\label{orthogonal-symplectic dual pairs}

We are aimed to prove the following proposition in this section.
\begin{prop}\label{typeos}
Conjecture \ref{Conjecture on Coincidence0} holds for all real  orthogonal-symplectic dual pairs and all complex  orthogonal-symplectic dual pairs.
\end{prop}

In this section, assume that  $D=\mathbb R$ or $\mathbb C$, and let $U$ be a finite-dimensional  symmetric bilinear space over $D$.   For each $n\geq 0$, let $V_n$ be a symplectic space over $D$ of dimension $2n$. Denote by $\mathrm G_n(U)$ the double cover of $\mathrm G(U)$ associated to the pair $(\mathrm G(U), \mathrm G(V_n))$. After twisting by an appropriate character, the representation of $\mathrm{G}_n(U)$ on $\omega_{\mathrm{G}(U),\mathrm{G}(V_n)}$ descents to a representation of $\mathrm{G}(U)$, which is denoted by $\omega_{U,n}$, so that
\[
  \mathrm{Hom}_{\mathrm G(U) \ltimes (U\otimes X_n)}(\omega_{U,n}, \mathbb C)\neq 0, \quad \textrm{for all Lagrangian subspace $X_n$ of $V_n$.}
\]
Here $U\otimes X_n$ is viewed as a subgroup of the Heisenberg group attached to the real symplectic space $U\otimes V_n$, and $\mathbb C$ stands for the trivial representation of $\mathrm G(U) \ltimes (U\otimes X_n)$. The reader is referred to \cite{Mo}, \cite{AB1} and \cite{AB2} for details about theta correspondence for (real or complex) orthogonal-symplectic dual pairs.

Fix a Cartan involution on $\mathrm{End}_D(U)$ which commutes with the adjoint involution of  the symmetric bilinear form on $U$, and denote by $\mathrm{K}(U)$ the corresponding maximal compact subgroup of $\mathrm{G}(U)$. Likewise, for each $n\geq 0$, fix a Cartan involution on $\mathrm{End}_D(V_n)$ which commutes with the adjoint involution of  the symplectic form on $V_n$. As in \eqref{productf}, we get a Cartan involution on $\mathrm{End}_{\mathbb R}(U\otimes V_n)$. Then as in the Introduction, we have a module  $\omega_{\mathrm{G}(U),\mathrm{G}(V_n)}^\mathrm{alg}$.  As in the smooth case, after twisting it by an appropriate character, we get a $(\mathfrak g(U), \mathrm{K}(U))$-module, to be denoted by  $\omega_{U,n}^\mathrm{alg}$. Here $\mathfrak g(U)$ denotes the complexified Lie algebra of $\mathrm{G}(U)$.

Let $\pi\in \mathrm{Irr}(\mathrm{G}(U))$. Define its first occurrence index
\[
  \mathrm{n}(\pi):= \min\{n\geq 0|  \mathrm{Hom}_{\mathrm{G}(U)}(\omega_{U,n},\pi)\neq 0\}.
\]
Define the algebraic analogue
 \[
  \mathrm{n}'(\pi):= \min\{n\geq 0|  \mathrm{Hom}_{\mathfrak g(U), \mathrm{K}(U)}(\omega_{U,n}^\mathrm{alg},\pi^\mathrm{alg})\neq 0\},
\]
where $\pi^\mathrm{alg}$ denotes the $(\mathfrak{g}(U), \mathrm{K}(U))$-module of $K(U)$-finite vectors in $\pi$. In view of  Kudla's persistence principle (both in the smooth case and the algebraic case), in order to prove Proposition \ref{typeos}, it suffices to show that
\[
   \mathrm{n}(\pi)= \mathrm{n}'(\pi).
\]

Let $\mathrm{sgn}$ denote the sign character of $\mathrm{G}(U)$.

\begin{lem}\label{nsgn}
The equality
\[
   \mathrm{n}'(\mathrm{sgn})= \dim U
   \]
   holds.
\end{lem}
\begin{proof}
 See \cite[Proposition 1.4]{AB1} and \cite[Proposition 2.10]{Pa3}.
\end{proof}

By a well-know argument (\emph{cf}. \cite[Section 2.1]{SZ}), Lemma  \ref{nsgn} implies
the following lemma.
\begin{lem}
The inequality
\begin{equation}\label{cons1}
   \mathrm{n}'(\pi)+   \mathrm{n}'(\mathrm{\pi \otimes sgn})\geq \dim U
\end{equation}
holds.
\end{lem}

On the other hand, by the conservation relation (\cite[Theorem 7.1]{SZ}), we have  that
\begin{equation}\label{cons2}
  \mathrm{n}(\pi)+   \mathrm{n}(\mathrm{\pi\otimes sgn})=\dim U.
\end{equation}
Together with the obvious inequalities
\[
\mathrm{n}'(\pi)\leq \mathrm{n}(\pi)\quad\textrm{and}\quad  \mathrm{n}'(\mathrm{\pi\otimes sgn})\leq \mathrm{n}(\mathrm{\pi \otimes sgn}),
\]
 \eqref{cons1} and \eqref{cons2} imply that $ \mathrm{n}(\pi)= \mathrm{n}'(\pi)$. This proves Proposition \ref{typeos}.

\section{Unitary dual pairs}\label{seunit}

For the proof of Theorem \ref{Main Theorem}, it remains to show the following proposition.

\begin{prop}\label{typeun}
Conjecture \ref{Conjecture on Coincidence0} holds for all unitary dual pairs.
\end{prop}

In this section, let $U$ be a complex skew-Hermitian space so that $\mathrm{G}(U)$ is a unitary group.  Let $\delta\in \mathbb Z/2\mathbb Z$. Define
\[
  \mathrm G_\delta (U):=\left\{
              \begin{array}{ll}
               \mathrm{G}(U)\times \{1,-1\},\quad & \textrm{if }\delta=0;\\
              \{(g,t)\in \mathrm{G}(U)\times \mathbb C^\times\mid \det g=t^2\}, \quad& \textrm{if } \delta=1,
              \end{array}
              \right.
\]
which is a double cover of $\mathrm G (U)$.
 For all $p,q\geq 0$, let $V_{p,q}$ be a Hermitian space of signature $(p,q)$. Assume that $p+q$ has parity $\delta$. Note that the double cover of $\mathrm{G}(U)$ associated to the dual pair $(\mathrm{G}(U), \mathrm{G}(V_{p,q}))$ is canonically isomorphic to $\mathrm G_\delta(U)$.
Thus $\omega_{\mathrm{G}(U), \mathrm{G}(V_{p,q})}$ is a representation of $\mathrm G_\delta(U)$. The reader is referred to \cite[Section 1]{Pa1} for details about the double covers associated to unitary dual pairs.

As in the last section, fix a Cartan involution on $\mathrm{End}_\mathbb C(U)$ which commutes with the adjoint involution of  the skwe-Hermitian form on $U$, and denote by $\mathrm K_\delta(U)$ the corresponding maximal compact subgroup of $\mathrm G_\delta(U)$. Likewise, fix a Cartan involution on $\mathrm{End}_\mathbb C(V_{p,q})$ which commutes with the adjoint involution of  the Hermitian  form on $V_{p,q}$. As in \eqref{productf}, we get a Cartan involution on $\mathrm{End}_{\mathbb R}(U\otimes V_{p,q})$. Then we define $\omega_{\mathrm{G}(U),\mathrm{G}(V_{p,q})}^\mathrm{alg}$ as in the Introduction.

Let $\pi\in \mathrm{Irr}^{\mathrm{gen}}(G_\delta(U))$. For  every integer $t$ with parity $\delta$, define the first occurrence index
\[
  \mathrm n_t(\pi):=\min\{ p+q\mid p,q\geq 0, \,p-q=t, \,\mathrm{Hom}_{\mathrm G_\delta(U)}(\omega_{\mathrm{G}(U), \mathrm{G}(V_{p,q})}, \pi)\neq 0\}.
\]
Define its algebraic analogue
 \[
  \mathrm n'_t(\pi):=\min\{ p+q\mid p,q\geq 0, \,p-q=t, \,\mathrm{Hom}_{\mathfrak g(U), \mathrm K_\delta(U)}(\omega_{\mathrm{G}(U), \mathrm{G}(V_{p,q})}^{\mathrm{alg}}, \pi^{\mathrm{alg}})\neq 0\}.
\]
Here and as before, $\frak g(U)$ denotes the complexified Lie algebra of $\mathrm{G}(U)$, and $\pi^{\mathrm{alg}}$ denotes the $(\mathfrak g(U), \mathrm K_\delta(U))$-module of $\mathrm K_\delta(U)$-finite vectors in $\pi$.

Recall the following fact of conservation relations from \cite[Theorem 7.6]{SZ}.
\begin{lem}\label{Conservation Relations for Unitary Dual Pairs}
There are two distinct integers $t_1$ and $t_2$, both of parity $\delta$,  such that
\begin{equation}\label{consu1}
  \mathrm{n}_{{t}_{1}}(\pi)+\mathrm{n}_{{t}_{2}}(\pi)=2\dim U+2.
\end{equation}
\end{lem}

In the algebraic setting, we need the following lemma.

\begin{lem}\label{Conservation Relations for Unitary Dual Pairs2}
Assume that $\delta=0$ and let $t$ be a nonzero even integer. Then
\begin{equation*}
  \mathrm{n}'_t(1_U)=2\dim U+|t|,
\end{equation*}
where $1_U\in \mathrm{Irr}^{\mathrm{gen}}(\mathrm G_\delta(U))$ denotes the representation with trivial action of $\mathrm{G}(U)\subset \mathrm G_\delta(U)$.
\end{lem}
\begin{proof}
See \cite[Lemma 3.1]{Pa2}.
\end{proof}

As in the last section,  Lemma \ref{Conservation Relations for Unitary Dual Pairs2}  implies the following lemma (\emph{cf}. the proof of \cite[Theorem 7.6]{SZ}).

\begin{lem}\label{lemalu}
\label{Algebraic Conservation Relations for Unitary Dual Pairs3}
Let $t_1$ and $t_2$ be two distinct integers, both of parity $\delta$. Then
\begin{equation}\label{consu2}
  \mathrm{n}'_{t_1}(\pi)+\mathrm{n}'_{t_2}(\pi)\geq 2\dim U+|t_1-t_2|.
\end{equation}
\end{lem}

We need the following lemma.

\begin{lem}\label{Algebraic Conservation Relations for Unitary Dual Pairs3}
Let $t$ be an integer with parity $\delta$. If $\mathrm{n}'_{t}(\pi)\leq \dim U$, then $\mathrm{n}_{t}(\pi)=\mathrm{n}'_{t}(\pi)$.

\end{lem}
\begin{proof}
Let $t_1$ and $t_2$ be as in Lemma \ref{Conservation Relations for Unitary Dual Pairs}. Together with the obvious inequalities
\[
\mathrm{n}_{t_1}'(\pi)\leq \mathrm{n_{t_1}}(\pi)\quad\textrm{and}\quad  \mathrm{n}_{t_2}'(\pi)\leq \mathrm{n}_{t_2}(\pi ),
\]
\eqref{consu1} and \eqref{consu2} imply that
\begin{equation}\label{nt}
  \mathrm{n}'_{t_1}(\pi)= \mathrm{n_{t_1}}(\pi)\quad\textrm{and}\quad  \mathrm{n}'_{t_2}(\pi)= \mathrm{n}_{t_2}(\pi ).
\end{equation}
Thus it suffices to show that $t=t_1$ or $t_2$. Assume this is not true. Then Lemma \ref{lemalu} implies that
\begin{equation}\label{nt2}
   \mathrm{n}'_{t}(\pi)+\mathrm{n}'_{t_i}(\pi)\geq 2\dim U+2,\quad i=1,2.
\end{equation}
Together with the inequality  $\mathrm{n}'_{t}(\pi)\leq \dim U$, \eqref{nt2} implies that
\[
  \mathrm{n}'_{t_1}(\pi)+\mathrm{n}'_{t_2}(\pi)\geq 2\dim U+4.
\]
This contradicts \eqref{consu1} and \eqref{nt}.
\end{proof}

Finally,  we come to the proof of Proposition \ref{typeun}.  Assume that the dual pair $(G,G')=(\mathrm{G}(U), \mathrm{G}(V))$, where $V$ is a complex Hermitian space. Without loss of generality, assume that $\dim V\leq \dim U$.

Suppose that  $V=V_{p,q}$ and $p+q$ has parity $\delta$. Let $\pi\in \mathscr{R}^{\mathrm{alg}}(\widetilde{G},\omega_{G,G'})$. Then
\[
  \mathrm{n}'_{p-q}(\pi)\leq p+q\leq \dim U.
\]
Thus by Lemma \ref{Algebraic Conservation Relations for Unitary Dual Pairs3},
\[
 \mathrm{n}_{p-q}(\pi)= \mathrm{n}'_{p-q}(\pi)\leq p+q.
 \]
This is equivalent to saying that $\pi\in \mathscr{R}^{\infty}(\widetilde{G},\omega_{G,G'})$. Therefore
\[
\mathscr{R}^{\mathrm{alg}}(\widetilde{G},\omega_{G,G'})\subset  \mathscr{R}^{\infty}(\widetilde{G},\omega_{G,G'})
\]
 and Proposition \ref{typeun} follows.

\end{document}